\documentclass[12pt,twoside]{article}
\usepackage{paralist}
\usepackage{marginnote}
\usepackage{tikz-cd}
\usepackage{amsfonts}
\usepackage{ntheorem}
\usepackage{amsrefs}
\usepackage{amssymb}
\usepackage{amsmath}
\usepackage{graphicx}
\usepackage{varwidth}
\usepackage{hyperref}
\usepackage{fancyhdr}
\usepackage{stmaryrd}
\usepackage{multirow}
\usepackage{blkarray}
\usepackage[autostyle]{csquotes}
\usepackage[mathscr]{euscript}
\hypersetup{
    colorlinks=true,
    linkcolor=black,
    filecolor=magenta,      
    urlcolor=red,
    citecolor=black
}
\DeclareMathAlphabet{\mathcalligra}{T1}{calligra}{c}{h}
\providecommand{\U}[1]{\protect\rule{.1in}{.1in}}
\newtheorem{theorem}{Theorem}
\newtheorem{proposition}[theorem]{Proposition}
\newtheorem{corollary}[theorem]{Corollary}
\newtheorem{lemma}[theorem]{Lemma}

\newtheorem*{question}{Question}
\newtheorem{remark}[theorem]{Remark}

\newcommand{\Ric}{\mathrm{Ric}}
\newcommand{\tr}{\mathrm{trace}}
\newcommand{\scal}{\mathrm{scal}}
\newcommand{\R}{\mathbb{R}}

\newcommand{\Z}{\mathbb{Z}}
\newcommand{\F}{\mathbb{F}}
\newcommand{\Vol}{\mathrm{Vol}}

\newcommand{\Hom}{\mathrm{Hom}}
\newcommand\Sp{\mathbb{S}}
\newcommand\Sh{\mathcal{L}}

\newcommand\bea{\begin{eqnarray*}}
\newcommand\eea{\end{eqnarray*}}
\newcommand\be{\begin{equation}}
\newcommand\ee{\end{equation}}

\newcommand{\mb}{\mathbb}

\newcommand{\po}{{\hspace*{-1ex}}{\bf .  }}

\def\<{\langle}
\def\>{\rangle}
\newcommand\qed{\ifhmode\unskip\nobreak\fi\ifmmode\ifinner\else
\hskip5 pt \fi\fi\hbox{\hskip5 pt \vrule width4 pt height6 pt
depth1.5 pt \hskip 1pt }}

\begin{document}

\title{Almost conformally flat hypersurfaces}
\author{Christos-Raent Onti and Theodoros Vlachos}
\date{}

\maketitle
\renewcommand{\thefootnote}{\fnsymbol{footnote}} 
\renewcommand{\thefootnote}{\arabic{footnote}} 

\renewcommand{\thefootnote}{\fnsymbol{footnote}} 
\renewcommand{\thefootnote}{\arabic{footnote}}

\begin{abstract}
We prove a universal lower bound for the $L^{n/2}$-norm of the 
Weyl tensor in terms of the Betti numbers for compact $n$-dimensional 
Riemannian manifolds that are conformally immersed as hypersurfaces in
the Euclidean space. As a consequence, we 
determine the homology of almost conformally flat hypersurfaces. Furthermore, 
we provide a necessary condition for a compact Riemannian manifold to admit 
an isometric minimal immersion as a hypersurface in the round sphere and extend a
result due to Shiohama and Xu \cite{SX} for compact hypersurfaces in any space form.
\end{abstract}

\section{Introduction}

The investigation of curvature and topology of Riemannian manifolds 
or submanifolds is one of the basic problems in global differential geometry. 
The sphere theorem is an important result in this direction, in the 
framework of Riemannian geometry. 

In the realm of conformal geometry, the fundamental tensor is the Weyl tensor 
and his role is similar to the one of the curvature tensor in Riemannian geometry.
Several authors have worked on the question how certain conditions on the Weyl tensor affect
the geometry and the topology of Riemannian manifolds (cf. \cite{Catino,DuNo}).
Schouten's theorem asserts that the vanishing of the Weyl tensor of a Riemannian $n$-manifold $M^n$
is equivalent to the fact that $M^n$ is conformally flat, i.e., locally is conformally diffeomorphic 
to an open subset of the Euclidean space $\R^{n}$, with the canonical metric, if $n\geq 4$.
The $L^{n/2}$-norm of the Weyl tensor, which is a
conformal invariant, measures how far a compact Riemannian manifold deviates 
from being conformally flat. There are plenty of papers that investigate the effect of restrictions on
the $L^{n/2}$-norm of the Weyl tensor to both geometric and topological properties 
(cf. \cite{Singer,Gursky,ItSa,ABKS,SeH,Listing}).

Cartan \cite{ECartan} initiated the investigation of conformally flat manifolds from the submanifold point
of view. Moore \cite{JDM2} and later do Carmo, Dajczer and Mercuri \cite{MMF} studied 
conformally flat submanifolds of the Euclidean space (see also \cite{Pinkall,DaFlo}). 
In particular, in \cite{MMF} the authors studied the case of hypersurfaces in $\R^{n+1}$ and they proved that diffeomorphically 
every such hypersurface $M^n$ is a sphere $\Sp^n$ with $\beta_1(M^n;\Z)$ handles attached, 
where $\beta_1(M^n;\Z)$ is the first Betti number of $M^n$. Moreover, they showed that 
geometrically every such hypersurface locally consists of nonumbilic submanifolds of $\R^{n+1}$ 
that are foliated by complete round $(n-1)$-spheres and are joined through their boundaries to 
the following three types of umbilic submanifolds of $\R^{n+1}$: $(i)$ an open piece of an 
$n$-sphere or an $n$-plane bounded by round $(n-1)$-sphere, $(ii)$ a round $(n-1)$-sphere, 
$(iii)$ a point.

\medskip

It is therefore natural to address the following
  
\begin{question}\po
Let $(M^n,g)$ be a compact $n$-dimensional Riemannian manifold that admits 
a conformal immersion in $\R^{n+1}$. What can be said about
the topology of $M^n$ if $(M^n,g)$ is almost conformally flat, in the sense that 
the $L^{n/2}$-norm of the Weyl tensor is sufficiently small?
\end{question}

The aim of the paper is to provide an answer to the above question. In particular, we give 
a universal lower bound for the $L^{n/2}$-norm of the Weyl tensor in terms of the Betti numbers 
for compact $n$-dimensional Riemannian manifolds that are conformally immersed in $\R^{n+1}$. 
As a consequence, we are able to determine the homology of compact almost conformally flat 
hypersurfaces. 
Throughout the paper, all manifolds under consideration are assumed to be without boundary, 
connected and oriented. Our main result is stated as follows.

\begin{theorem}\po\label{main}
Given $n\geq 4$, there exists a positive constant $c(n)$, depending only on $n$, 
such that if $(M^n,g)$ is a compact $n$-dimensional Riemannian manifold that 
admits a conformal immersion in $\mb{R}^{n+1}$, then 
the Weyl tensor associated to $g$ satisfies
\be\label{W1}
\int_{M^n}\Vert \mathcal{W}_g \Vert^{n/2}\ dM_g \geq c(n)\sum_{i=2}^{n-2}\beta_i(M^n;\F),
\ee
where $\beta_i(M^n;\F)$ is the $i$-th Betti number of $M^n$ over an arbitrary coefficient 
field $\mathbb{F}$. In particular, if 
\be\label{W4}
\int_{M^n}\Vert \mathcal{W}_g \Vert^{n/2}\ dM_g<c(n)
\ee
then $M^n$ has the homotopy type of a CW-complex with no cells of dimension $i$ for 
$2\leq i\leq n-2$ and the fundamental group $\pi_1(M^n)$ is a free group on $\beta_1(M^n;\Z)$ 
generators. Moreover, if $\pi_1(M^n)$ is finite then $M^n$ is homeomorphic to the 
sphere $\Sp^n$.
\end{theorem}

Thus, in the case where (\ref{W4}) is satisfied, the homology groups of $M^n$ must 
satisfy the condition $H_i(M^n;\F)=0$ for all $2\leq i\leq n-2$, where $\F$ is any coefficient field. 

The assumption on the codimension in Theorem \ref{main} is essential, as the following 
example shows. We consider the manifold 
$M^n=\Sp^{1}(1) \times\Sp^{1}(1)\times\Sp^{n-2}(r),\ n\geq 4,$ equipped with the product metric 
$g$, where $\Sp^{n-2}(r)$ is the $(n-2)$-dimensional round sphere of radius $r$. Since $M^n$ 
is  isometrically immersed into the sphere $\Sp^{n+2}(\sqrt{2+r^2})$ in the obvious way, it follows 
that $M^n$ admits a conformal immersion into $\mathbb{R}^{n+2}$. A long but straightforward 
computation yields that 
$$\int_{M^n}\Vert \mathcal{W}_g \Vert^{n/2}\ dM_g=\frac{a(n)}{r^2},$$ where $a(n)$ is a positive 
constant depending only on $n$. Since the second Betti number of $M^n$ is non-zero, 
over an arbitrary coefficient field $\F$, it follows that $M^n$ does not satisfy inequality (\ref{W1}) for 
$r$ large enough.

\begin{remark}\po\label{rem}\rm{
In Theorem \ref{main}, the ambient space $\R^{n+1}$ can be replaced by the round sphere $\Sp^{n+1}$ or the hyperbolic 
space $\mathbb{H}^{n+1}$. Indeed, this follows from the fact that $\Sp^{n+1}\smallsetminus\{\text{point}\}$ and 
$\mathbb{H}^{n+1}$ are conformally equivalent to the Euclidean space and  that the 
$L^{n/2}$-norm of the Weyl tensor is conformally invariant.
}
\end{remark}

Chern and Simons provided in \cite[Theorem 6.4, p. 65]{ChSim} a necessary condition for a compact $3$-dimensional 
Riemannian manifold to admit a conformal immersion into $\R^4$.
Theorem \ref{main} allows us to give another such condition.

\begin{corollary}\po
Let $(M^n,g),\ n\geq 4,$ be a compact $n$-dimensional Riemannian manifold. A necessary condition
that $M^n$ admit a conformal immersion in $\mb{R}^{n+1}$ is inequality (\ref{W1}).
\end{corollary}

There is an abundance of Riemannian manifolds that do not satisfy inequality (\ref{W1}), and therefore
do not admit a conformal immersion as hypersurfaces in $\R^{n+1}$. For 
instance, let $M^n=N^{m}\times\mb{S}^{n-m}(r),\  \ 2\leq m\leq n-2,$ equipped with the 
product metric $g$, where $N^{m}$ is a compact flat $m$-dimensional Riemannian manifold. 
A long but straightforward computation yields that 
$$\int_{M^n}\Vert \mathcal{W}_g \Vert^{n/2}\ dM_g=\frac{a(n,m)}{r^m}\Vol(N^m),$$ where $a(n,m)$ is 
a positive constant depending only on $n$ and $m$. Since, the $(n-m)$-th Betti number of $M^n$ is non-zero, over an 
arbitrary coefficient field $\F$, we obtain that $M^n$ does not satisfy (\ref{W1}) for $r$ large enough.

\smallskip

Using Theorem \ref{main}, we are able to prove the following result for compact minimal hypersurfaces in spheres.

\begin{theorem}\po\label{MHS}
Given $n\geq 4$, there exists a positive constant $c_1(n)$, depending only on $n$, such that if 
$(M^n,g)$ is a compact $n$-dimensional Riemannian manifold that admits an isometric 
minimal immersion into the unit $(n+1)$-dimensional sphere $\Sp^{n+1}$, then
$$
\int_{M^n} \Vert \Ric_g-(n-1) g\Vert^{n/2}\ dM_g\geq c_1(n)\sum_{i=2}^{n-2}\beta_i(M^n;\mathbb{F}),
$$
where $\Ric_g$ is the Ricci tensor of $(M^n,g)$. In particular, if 
$$
\int_{M^n} \Vert \Ric_g-(n-1) g \Vert^{n/2} \  dM_g<c_1(n),
$$
then $M^n$ has the homotopy type of a CW-complex with no cells of dimension $i$ for 
$2\leq i\leq n-2$ and the fundamental group $\pi_1(M^n)$ is a free group on $\beta_1(M^n;\Z)$ 
elements. Moreover, if $\pi_1(M^n)$ is finite then $M^n$ is homeomorphic to the 
sphere $\Sp^n$.

\end{theorem}

As an immediate application of Theorem \ref{MHS}, we obtain an obstruction for  a Riemannian metric 
to be realized on a compact minimal hypersurface of the sphere.

\begin{corollary}\po
A compact $n$-dimensional Riemannian manifold $(M^n,g),\ n\geq 4,$ that satisfies
$$
\int_{M^n} \Vert \Ric_g-(n-1) g\Vert^{n/2}\ dM_g < c_1(n)\sum_{i=2}^{n-2}\beta_i(M^n;\mathbb{F}),
$$
cannot admit an isometric minimal immersion into the unit sphere $\Sp^{n+1}$.
\end{corollary}

Shiohama and Xu \cite{SX} gave a lower bound in terms of the Betti numbers for the $L^{n/2}$-norm 
of the $(0,4)$-tensor $R_g-\big(\scal_g/n(n-1)\big)R_1$ of compact hypersurfaces in $\R^{n+1}$.
Here, $R_g$ and $\scal_g$ denote the $(0,4)$-curvature 
tensor and the scalar curvature of the induced metric $g$ respectively and $R_1=(1/2)g\varowedge g$, where $\varowedge$ 
stands for the Kulkarni-Nomizu product. The $L^{n/2}$-norm of this tensor 
measures how far the Riemannian manifold deviates from being a space form.
Their proof strongly uses the fact that the ambient space is the Euclidean one. 
Using Theorem \ref{main}, we extend their result for compact 
hypersurfaces in spheres or in the hyperbolic space.

\begin{theorem}\po \label{Th3}
If a compact $n$-dimensional Riemannian manifold $(M^n,g),\ n\geq 4,$ admits an isometric immersion
into the sphere $\Sp^{n+1}$ or the hyperbolic space $\mathbb{H}^{n+1}$, then
$$
\int_{M^n} \Big\Vert R_g-\frac{\scal_g}{n(n-1)}R_1\Big\Vert^{n/2}\ dM_g\geq c(n)\sum_{i=2}^{n-2}\beta_i(M^n;\mathbb{F}),
$$
where $c(n)$ is the constant that appears in Theorem \ref{main}.
In particular, if 
$$
\int_{M^n} \Big\Vert R_g-\frac{\scal_g}{n(n-1)}R_1 \Big\Vert^{n/2} \  dM_g<c(n),
$$
then $M^n$ has the homotopy type of a CW-complex with no cells of dimension $i$ for 
$2\leq i\leq n-2$ and the fundamental group $\pi_1(M^n)$ is a free group on $\beta_1(M^n;\Z)$ 
elements. Moreover, if $\pi_1(M^n)$ is finite then $M^n$ is homeomorphic to the 
sphere $\Sp^n$.
\end{theorem}

We notice that related results have been obtained in \cite{SX1,SX2, Vlachos}.

\section{Algebraic auxiliary results}

This section is devoted to some algebraic results that are crucial for the proofs. 
Let $V$ be a finite dimensional real vector space equipped with a positive definite inner 
product $\langle \cdot,\cdot\rangle$. We denote by $\Hom(V\times V,\mb{R})$ the space 
of all bilinear forms and by $\mathrm{Sym}(V\times V,\mb{R})$ its subspace that consists 
of all symmetric bilinear forms.

The {\it Kulkarni-Nomizu product} of two bilinear forms $\phi,\psi\in \mathrm{Hom}(V\times V,\mb{R})$ 
is the $(0,4)$-tensor $\phi\varowedge \psi\colon V\times V\times V\times V \rightarrow \mb{R}$ defined by 
\bea
\phi\varowedge \psi(x_1,x_2,x_3,x_4) &=& \phi(x_1,x_3) \psi(x_2,x_4)+\phi(x_2,x_4) \psi(x_1,x_3) \\
 & &-\phi(x_1,x_4) \psi(x_2,x_3)-\phi(x_2,x_3)\psi(x_1,x_4).
\eea

Let $W$ be a finite dimensional real vector space equipped with a nondegenerate inner product 
$\langle\cdot ,\cdot \rangle_W$. Using the inner product of $W$, we extend the {\it Kulkarni-Nomizu 
product} to $W$-valued bilinear forms $\beta,\gamma\in\Hom(V\times V,W)$, as the $(0,4)$-tensor 
$\beta\varowedge \gamma\colon V\times V\times V\times V \rightarrow \mb{R}$ defined by 
\bea
\beta\varowedge\gamma(x_1,x_2,x_3,x_4) \!\!\!&=&\!\!\! 
\langle\beta(x_1,x_3),\gamma(x_2,x_4) \rangle_W 
-\langle\beta(x_1,x_4),\gamma(x_2,x_3) \rangle_W \\
\!\!\!& &\!\!\! +\langle\beta(x_2,x_4),\gamma(x_1,x_3) \rangle_W-
\langle\beta(x_2,x_3),\gamma(x_1,x_4) \rangle_W.
\eea

A bilinear form $\beta\in \mathrm{Hom}(V\times V,W)$ is called {\it{flat}} with respect 
to the inner product $\langle\cdot ,\cdot \rangle_W$ of $W$ if
$$\langle\beta(x_1,x_3),\beta(x_2,x_4)\rangle_W
-\langle\beta(x_1,x_4),\beta(x_2,x_3)\rangle_W=0$$ 
for all $x_1,x_2,x_3,x_4\in V$, or equivalently if $\beta\varowedge\beta=0.$

Associated to each bilinear form $\beta$ is the {\it nullity space} $\mathcal{N}(\beta)$ 
defined by 
$$
\mathcal{N}(\beta)=\left\{x\in V\ :\ \beta(x,y)=0\ \ \text{for all}\ \ y\in V\right\}.
$$

We need the following result on flat bilinear forms, which is due to Moore 
(cf. \cite[Proposition 2, p. 93]{JDM2}).

\begin{lemma}\po\label{JDMLe}
Let $\beta\in\mathrm{Sym}(V\times V,W)$ be a flat bilinear form with respect to
 a Lorentzian inner product of $W$. If $\dim V>\dim W$ and $\beta(x,x)\neq 0$ 
 for all non-zero $x\in V$, then there is a non-zero isotropic vector $e\in W$ and 
 a bilinear form $\phi\in \mathrm{Sym}(V\times V, \R)$ such that 
 $\dim \mathcal{N}(\beta-e\phi)\geq \dim V-\dim W+2.$
\end{lemma}

We define the map 
${\sf W}\colon \mathrm{Sym}(V\times V,\R)\rightarrow\mathrm{Hom}(V\times V\times V\times V,\R)$ by 
$${\sf W}(\beta) = {\sf R}(\beta)-{\sf L}(\beta)\varowedge \langle\cdot,\cdot\rangle,$$ where 
$${\sf R}(\beta) = \frac{1}{2}\beta\varowedge\beta,\ \ 
{\sf L}(\beta)=\frac{1}{n-2}\Big({\sf Ric}(\beta)-\frac{{\sf scal}(\beta)}{2(n-1)}\langle\cdot,\cdot\rangle\Big),$$
$${\sf Ric}(\beta)(x,y)=\tr\ {\sf R}(\beta)(\cdot,x,\cdot,y),\ \ x,y\in V\ \ \text{and}\ \ {\sf scal}(\beta)=\tr \ {\sf Ric} (\beta).$$

To each $\beta\in \mathrm{Sym}(V\times V,\mb{R})$ we assign a self-adjoint endomorphism 
$\beta^\sharp$ of $V$ defined by
$$\langle \beta^\sharp(x),y\rangle= \beta(x,y),\ \ x,y\in V.$$

\begin{lemma}\po\label{Cartan}
Let $\dim V=n\geq 4$ and $\beta\in \mathrm{Sym}(V\times V, \mb{R})$. 
Then ${\sf W}(\beta)=0$ if and only if $\beta^\sharp$ has an eigenvalue 
of multiplicity at least $n-1$.
\end{lemma}

\begin{proof}
Let $\beta\in \mathrm{Sym}(V\times V, \R)$ be a symmetric bilinear form such that 
${\sf W}(\beta)=0$. We endow $\R^3$ with the  Lorentzian inner product $\langle\langle\cdot , \cdot \rangle\rangle$ given by
$$
\langle\langle (x_1,x_2,x_3), (y_1,y_2,y_3)  \rangle\rangle=x_1y_1+x_2y_3+x_3y_2,
$$ 
and define the symmetric bilinear form 
$\widetilde{\beta}\colon V\times V\rightarrow \R^3$ by 
$$\widetilde{\beta}(x,y)=\big(\beta(x,y),\langle x,y \rangle, -{\sf L}(\beta)(x,y)\big).$$
Since ${\sf W}(\beta)=0$ it follows that $\widetilde{\beta}$ is flat with respect to 
$\langle\langle\cdot , \cdot \rangle\rangle$. From Lemma \ref{JDMLe}, we know that 
there exists a non-zero isotropic vector $e=(t_1,t_2,t_3)\in\R^3$ and a symmetric bilinear form 
$\phi:V\times V\rightarrow \mathbb{R}$ such that 
$\dim\mathcal{N}(\widetilde\beta-\phi e)\geq n-1$. By setting 
$V_1=\mathcal{N}(\widetilde{\beta}-e\phi)$, we have that
$$
\widetilde{\beta}(x,y)=\phi(x,y) e,$$ or equivalently 
$$
\beta(x,y) = \phi(x,y)t_1, \ \ \langle x,y \rangle = \phi(x,y)t_2 \ \ 
\text{and}\ \ {\sf L}(\beta)(x,y)=-\phi(x,y) t_3
$$
 for all $x\in V_1\ \text{and}\ y\in V.$
Therefore, $$\beta(x,y) = \langle x,y \rangle \lambda$$ for all $x\in V_1\ \text{and}\ y\in V,$ 
where $\lambda=t_1/t_2$. Hence, $\lambda$ is an eigenvalue of $\beta^\sharp$ with multiplicity at least $n-1$.

Conversely, assume that 
$$
\beta^\sharp e_i=\lambda e_i,\ \ 1\leq i\leq n-1 \ \ 
\text{and}\ \ \beta^\sharp e_n=\mu e_n,
$$
 where $e_1,...,e_n$ is  an orthonormal basis of $V$. 
A long but straightforward computation then yields ${\sf W}(\beta)=0$. \qed
\end{proof}

\smallskip 

The following proposition is crucial for the proof of Theorem \ref{main}.

\begin{proposition}\po\label{mainlemma}
Given $n\geq 4$, there exists a positive constant $\varepsilon(n)$, depending only on $n$, 
such that the following inequality holds
$$
\Vert {\sf W}(\beta)\Vert^2\geq \varepsilon(n)\vert\det \beta^\sharp\vert^{4/n}
$$
 for all $\beta\in\mathrm{Sym}(V\times V,\mb{R})\smallsetminus (E_+\cup E_{-})$, 
 where $V$ is an $n$-dimensional vector space equipped with a positive definite inner 
 product, 
 $$
 E_{\pm}=\left\{\beta\in \mathrm{Sym}(V\times V,\mb{R})\ :\  \mathcal{E}_\pm(\beta)\geq n-1\right\},
 $$
and $\mathcal{E}_+(\beta)$ (respectively, $\mathcal{E}_{-}(\beta)$) is the number of positive (respectively, negative) 
eigenvalues of $\beta^\sharp$, each one counted with its multiplicity.
\end{proposition}

\begin{proof}
Let $\beta\in\mathrm{Sym}(V\times V,\mb{R})$ and let $e_1,\dots,e_n$ be an orthonormal 
basis of $V$ that diagonalizes $\beta^\sharp$ with corresponding eigenvalues 
$\lambda_1,\dots,\lambda_n$. By an easy computation we obtain
$$
\Vert {\sf W}(\beta)\Vert^2=4\sum_{i<j}\left(\lambda_i\lambda_j-
\frac{1}{n-2}\Big((\lambda_i+\lambda_j)\sum_{k=1}^n\lambda_k-
(\lambda_i^2+\lambda_j^2)-\frac{\sum_{k\neq l}\lambda_k\lambda_l}{n-1}\Big)\right)^2.
$$
We consider the functions $\phi,\psi\colon \R^n \rightarrow \mathbb{R}$ defined by
$$
\phi(x) = 4\sum_{i<j}\left(x_ix_j-\frac{1}{n-2}\Big(\sigma_1(x)(x_i+x_j)-(x_i^2+x_j^2)-
\frac{2\sigma_2(x)}{n-1}\Big)\right)^2\;\; \text{and} \;\; \psi(x)=\prod_{i=1}^n x_i,
$$  
where 
$$
\sigma_1(x)=\sum_{i=1}^nx_i,\ \sigma_2(x)=\sum_{i<j}x_ix_j\;\; \text{and}\;\; x=(x_1,\dots,x_n).
$$
In order to prove the desired inequality, it is sufficient to show that there exists a positive constant 
$\varepsilon(n)$, depending only on $n$, such that 
$$
\phi(x)\geq \varepsilon(n)\psi(x)\ \ \text{for all}\ \ x\in U_n,
$$
where $U_n=\R^n\smallsetminus (K^n_+\cup K^n_-)$ and $K^n_+$ (respectively, $K^n_{-}$) is the subset 
of points in $\R^n$ with at least $n-1$ positive (respectively, negative) coordinates.

At first we are going to prove that $\phi$ attains a positive minimum on the level set 
$$
\Sigma_n=\left\{x\in U_n\ :\ \psi(x)=\varepsilon\right\},
$$ 
where $\varepsilon =\pm1$. Since 
$\phi(\Sigma_n)$ is bounded from below, there exists a sequence $\{z_m\}$ in $\Sigma_n$ 
such that $\lim_{m\rightarrow\infty}\phi(z_m)=\inf\phi(\Sigma_n)\geq0.$ We write 
$z_m=\rho_m a_m$, where $\rho_m=\Vert z_m\Vert$ and $a_m$ lies in the unit $(n-1)$-dimensional sphere 
$\mb{S}^{n-1}\subset\mb{R}^n$. 

We claim that $\{z_m\}$ is bounded. Suppose, to the contrary, that there exists a subsequence 
of $\{z_m\}$, which by abuse of notation is again denoted  by $\{z_m\}$, such that 
$\lim_{m\rightarrow\infty}\rho_m=\infty.$ Since $a_m\in\Sp^{n-1}$ we may assume, by taking a 
subsequence if necessary, that $\lim_{m\rightarrow\infty}a_m=a\in\Sp^{n-1}$.  

From the homogeneity of $\phi$ and $\psi$ we obtain 
$$\phi(a_m)=\frac{\phi(z_m)}{\rho_m^4}\ \ \text{and}\ \ \rho_m=\frac{1}{\vert \psi(a_m) \vert^{1/n}}.$$
Hence, $\lim_{m\rightarrow\infty} \phi(a_m)=0$ and thus $\phi(a)=0$. Lemma \ref{Cartan} implies that
 at least $n-1$ coordinates of $a$ are equal. On the other hand, $a_m\in U_n$ and since $U_n$ is 
closed we have $a\in U_n$. Therefore, $n-1$ coordinates of $a$ vanish. After an eventual reenumeration, 
we may suppose that $a=(\varepsilon,0,\dots,0)$. We set 
$$
a_m=(a_{m,1},\dots,a_{m,n})\ \ 
\text{and}\ \  \eta_m=\Big(\sum_{i=2}^n a_{m,i}^2\Big)^{1/2}.
$$ 
Since $\psi(a_m)\neq 0$, we may write 
$$
(a_{m,2},\dots,a_{m,n})=\eta_m\theta_m,
$$ 
where 
$$
\theta_m=(\theta_{m,2},\dots,\theta_{m,n})
$$ 
lies in the unit 
$(n-2)$-dimensional sphere $\Sp^{n-2}\subset \R^{n-1}$.
Then from $\phi(z_m)= \rho_m^4\phi(a_m)$ we have that
$$
\phi(z_m) \geq 4\rho_m^4\sum_{2\leq i<j}\left(a_{m,i}a_{m,j}-
\frac{1}{n-2}\Big(\sigma_1(a_m)(a_{m,i}+a_{m,j})-(a_{m,i}^2+a_{m,j}^2)-
\frac{2\sigma_2(a_m)}{n-1}\Big)\right)^2.
$$
We observe that
 $$\sigma_1(a_m)=a_{m,1}+{\eta_m}\sum_{i=2}^n\theta_{m,i} \ \ \text{and}\ \
\sigma_2(a_m) =\eta_m\widetilde{\sigma}_2(\theta_m),
$$
where 
$$\widetilde{\sigma}_2(\theta_m)=a_{m,1}\sum_{j=2}^n \theta_{m,j}+\eta_m\sum_{2\leq i<j}\theta_{m,i}\theta_{m,j}.$$
Therefore, we have
\begin{equation}\label{P1}
\phi(z_m)\geq 4\rho_m^4 \eta_m^2\delta_m,
\end{equation} where 
$$\delta_m=\sum_{2\leq i<j}\left({\eta_m}\theta_{m,i}\theta_{m,j}-
\frac{1}{n-2}\Big(\sigma_1(a_m)(\theta_{m,i}+\theta_{m,j})-{\eta_m}(\theta_{m,i}^2+\theta_{m,j}^2)
-\frac{2\widetilde{\sigma}_2(\theta_m)}{n-1}\Big)\right)^2.$$
Moreover, we obtain
\bea
\rho_m^4\eta_m^2\!\!\!&=&\!\!\! \frac{\eta_m^2}{|\psi(a_m)|^{4/n}} \\[1mm] 
\!\!\!&=&\!\!\!  \frac{\eta_m^2}{|a_{m,1}|^{4/n} \big\vert\prod_{i=2}^n a_{m,i}\big\vert^{4/n}} \\[1mm] 
\!\!\!&=&\!\!\!  \frac{1}{|a_{m,1}|^{4/n} \ \eta_m^{\frac{2(n-2)}{n}} \big\vert\prod_{i=2}^n \theta_{m,i}\big\vert^{4/n}}.
\eea
By passing if necessary to a subsequence, we may assume that 
$$
\lim_{m\rightarrow \infty}\theta_m=(\bar{\theta}_2,\dots,\bar{\theta}_n)\in\mb{S}^{n-2}.
$$
Clearly the above yields $$\lim_{m\rightarrow\infty}\rho_m^4\eta_m^2=\infty.$$
Using 
$$
\lim_{m\rightarrow \infty}\sigma_1(a_m)=\varepsilon\ \ \text{and}\ \ 
\lim_{m\rightarrow \infty}\widetilde{\sigma}_2(\theta_m)= \varepsilon\sum_{k=2}^n\bar{\theta}_k,
$$ 
we find that
$$
\lim_{m\rightarrow \infty}\delta_m = 
\frac{1}{(n-2)^2}\sum_{2\leq i<j}\left(\bar{\theta}_i+\bar{\theta}_j-\frac{2\sum_{k=2}^n\bar{\theta}_k}{n-1}\right)^2.
$$

We claim that $\lim_{m\rightarrow \infty}\delta_m\neq0$. Arguing indirectly, assume that 
$\lim_{m\rightarrow \infty}\delta_m=0$. This implies that 
$\bar{\theta}_2=\cdots=\bar{\theta}_n\neq0,$ 
which contradicts for large $m$ the fact that $a_m\in U_n$.
Thus, by taking limits in (\ref{P1}) we reach a contradiction 
and this proves the claim that the sequence $\{z_m\}$ is bounded. 

By passing if necessary to a subsequence, we have 
$$
\lim_{m\rightarrow\infty}z_m=z\in \Sigma_n.
$$  
Since $\Sigma_n$ 
doesn't contain any zeros of $\phi$ it follows that
$$
\min\phi(\Sigma_n)=\lim_{m\rightarrow \infty}\phi(z_m)=\phi(z)>0.
$$ 
Hence the function $\phi$ attains a positive minimum $\varepsilon (n)=\phi(z)$ 
on $\Sigma_n$, which obviously depends only on $n$.

Now, let $x\in U_n$. Assume that $\psi(x)\neq 0$ and set 
$\widetilde{x}=x/|\psi(x)|^{1/n}$. Clearly $\widetilde{x}\in \Sigma_n$ 
and consequently $\phi(\widetilde{x})\geq \varepsilon(n)$. Since $\phi$ 
is homogeneous of degree $4$, the desired inequality is obviously fulfilled. 
In the case where $\psi(x)= 0$, the inequality is trivial. \qed
\end{proof}

\begin{remark}\po\label{constant}
\rm{
The constant $\varepsilon(n)$ that appears in Proposition \ref{mainlemma} is not computed 
explicitly here, although one can apply the Lagrange multiplier method to 
compute it. 
}
\end{remark}

\section{The proofs}

At first we recall some well known facts on the total curvature and how Morse 
theory provides restrictions on the Betti numbers.
Let $f\colon (M^n,g)\rightarrow \mb{R}^{n+k}$ be an isometric immersion of a 
compact, connected and oriented $n$-dimensional Riemannian manifold into 
the $(n+k)$-dimensional Euclidean space $\R^{n+k}$ equipped with the usual 
inner product $\langle\cdot,\cdot\rangle$. The normal bundle of $f$ is given by 
$$N_fM=\left\{(p,\xi)\in f^*(T\mb{R}^{n+k})\ :\ \xi\perp df_p(T_pM) \right\}$$ and 
the corresponding unit normal bundle is defined by 
$$UN_f=\left\{(p,\xi)\in N_fM\ :\ \Vert\xi\Vert=1\right\},$$
where $f^*(T\mb{R}^{n+k})$ is the induced bundle of $f$ over $M$.

The {\it generalized Gauss map} $\nu\colon UN_f\rightarrow \mb{S}^{n+k-1}$ is 
defined by $\nu(p,\xi)=\xi$, where $\Sp^{n+k-1}$ is the unit $(n+k-1)$-dimensional 
sphere of $\R^{n+k}$. For each $u\in\mb{S}^{n+k-1}$, we consider the height function 
$h_u\colon M^n\rightarrow \mb{R}$ defined by $h_u(p)=\langle f(p),u\rangle,\ p\in M^n$. 
Since $h_u$ has a degenerate critical point if and only if $u$ is a critical point of the 
generalized Gauss map, by Sard's theorem there exists a subset $E\subset\Sp^{n+k-1}$ 
of measure zero such that $h_u$ is a Morse function for all $u\in\Sp^{n+k-1}\smallsetminus E$. 
For each $u\in \Sp^{n+k-1}\smallsetminus E$, we denote by $\mu_i(u)$ the number of critical 
points of $h_u$ of index $i$. We also set $\mu_{i}(u)=0$ for any $u\in E$. Following Kuiper 
\cite{Kuiper}, we define the {\it total curvature of index $i$ of $f$} by 
$$
\tau_i(f)=\frac{1}{\mathrm{Vol}(\mb{S}^{n+k-1})}\int_{\mb{S}^{n+k-1}}\mu_i(u)\ d\mb{S},
$$
where $d\Sp$ denotes the volume element of the sphere $\Sp^{n+k-1}$.

Let $\beta_i(M^n; \F)=\dim_\F H_i(M^n;\F)$ be the $i$-th Betti number of $M^n$ over an arbitrary coefficient field 
$\F$. Then, due to weak Morse inequalities \cite[Theorem 5.2, p. 29]{Milnor} we have that $\mu_i(u)\geq \beta_i(M^n;\F)$, 
for all $u\in \Sp^{n+k-1}\smallsetminus E$. By integrating over $\Sp^{n+k-1}$, we obtain 
\begin{equation}\label{TC}
\tau_i(f)\geq \beta_i(M^n;\F).  
\end{equation}

For each $(p,\xi)\in UN_f$, we denote by $A_{\xi}$ the shape operator  of $f$ in the direction 
$\xi$ which is given by $$g (A_{\xi} X,Y)=\langle \alpha_f(X,Y),\xi\rangle,$$ where $X,Y$ are 
tangent vector fields of $M^n$ and $\alpha_f$ is the second fundamental form of the immersion 
$f$ viewed as a section of the vector bundle $\mathrm{Hom}(TM\times TM,N_f M)$. There is a natural 
volume element $d\Sigma$ on the unit normal bundle $UN_f$. In fact, if $dV$ is a $(k-1)$-form on 
$UN_f$ such that its restriction to a fiber of the unit normal bundle at $(p,\xi)$ is the volume
element of the unit $(k-1)$-sphere of the normal space of $f$ at $p$,
then $d\Sigma=dM_g\wedge dV$, where $dM_g$ is the volume element of $M^n$ with 
respect to the metric $g$.

Shiohama and Xu considered for each $0\leq i\leq n$ the subset $U^iN_f$ of the unit normal 
bundle of $f$ defined by 
$$
U^iN_f=\left\{(p,\xi)\in UN_f\ :\ \mathrm{Index}(A_\xi)=i\right\}
$$ 
and proved 
(cf.  \cite[Lemma, p. 381]{SX}) that 
\begin{equation}\label{ShXu}
\int_{U^iN_f}|\mathrm{det}A_\xi |\ d\Sigma=\int_{\mb{S}^{n+k-1}}\mu_i(u)\ d\Sp.
\end{equation}

We recall that the $(0,4)$-Riemann curvature tensor $R_g$ of $(M^n,g)$ is related to the 
second fundamental form of $f$ via the Gauss equation 
$$R_g(X,Y,Z,W)=\langle \alpha_f(X,Z),\alpha_f(Y,W)\rangle-
\langle\alpha_f(X,W),\alpha_f(Y,Z)\rangle,$$ where $X,Y,Z,W$ are tangent vector fields of $M^n$. 
In terms of the Kulkarni-Nomizu product, the Gauss equation is written equivalently as 
$$R_g=\frac{1}{2}\alpha_f\varowedge\alpha_f.$$
On the other hand, $R_g$ admits the following orthogonal decomposition (cf. \cite[Chapter 8D]{Kun})
\begin{equation}\label{Rdec}
R_g=\mathcal{W}_g+\frac{\scal_g}{2n(n-1)}g\varowedge g+
\frac{1}{n-2}\Big(\Ric_g-\frac{\scal_g}{n}g\Big)\varowedge g,
\end{equation}
which is in addition irreducible with respect to the (simultaneous) action 
of the orthogonal group $\mathsf{O}(n)$ on the four arguments of $R_g$. Equation 
(\ref{Rdec}) is equivalent to 
\be \label{WSh}
R_g=\mathcal{W}_g+\Sh_g\varowedge g,
\ee
where $$\Sh_g=\frac{1}{n-2}\Big(\Ric_g-\frac{\scal_g}{2(n-1)}g\Big).$$
The $(0,4)$-tensor $\mathcal{W}_g$ is the {\it Weyl tensor} and the $(0,2)$-tensor 
$\Sh_g$ is the {\it Schouten tensor} of the Riemannian manifold $(M^n,g)$.

\medskip

\noindent\emph{Proof of Theorem \ref{main}:} 
Let $f\colon (M^n,g)\rightarrow \mathbb{R}^{n+1}$ be a conformal immersion with 
unit normal bundle $UN_f$ and shape operator $A_\xi$ with respect to $\xi$, where 
$(p,\xi)\in UN_f$. Using the Gauss equation, it follows that the Weyl tensor $\mathcal{W}_{\tilde{g}}$ 
with respect to the metric $\tilde{g}$ induced by $f$  is given by 
$$
\mathcal{W}_{\tilde{g}}(p)=\mathsf{W}(\alpha_f(p)).
$$ 
From Proposition \ref{mainlemma} 
we obtain 
$$
\Vert \mathcal{W}_{\tilde{g}}(p)\Vert^2\geq \varepsilon(n)\vert \det A_\xi(p)\vert^{4/n}
$$ 
for all $(p,\xi)\in U^iN_f$ and $2\leq i\leq n-2$. By integrating over $UN_f$, we have that
\begin{equation}\label{W2}
\int_{UN_f}\Vert\mathcal{W}_{\tilde{g}}\Vert^{n/2} \ d\Sigma\geq 
(\varepsilon(n))^{n/4}\sum_{i=2}^{n-2}\int_{U^iN_f} |\det A_\xi| \ d\Sigma.
\end{equation}
Using (\ref{ShXu}), it follows that 
$$
\sum_{i=2}^{n-2}\int_{U^iN_f} |\det A_\xi| \ d\Sigma=
\sum_{i=2}^{n-2}\int_{\Sp^n}\mu_i(u) \ d\Sp=\Vol(\Sp^{n})\sum_{i=2}^{n-2}\tau_i(f),
$$
where $d\Sp$ is the volume element of the unit $n$-dimensional sphere $\Sp^n$. 
Therefore, from (\ref{W2}) and bearing in mind (\ref{TC}), we obtain
\begin{equation}\label{W3}
\int_{UN_f}\Vert\mathcal{W}_{\tilde{g}}\Vert^{n/2} \ d\Sigma \geq 
(\varepsilon(n))^{n/4} \Vol(\Sp^{n})\sum_{i=2}^{n-2}\tau_i(f) \geq 
(\varepsilon(n))^{n/4} \Vol(\Sp^{n})\sum_{i=2}^{n-2}\beta_i(M^n;\F).
\end{equation}
Observe that 
$$\int_{M^n}\Vert \mathcal{W}_{\tilde{g}} \Vert^{n/2}\ dM_{\tilde{g}}=
\frac{1}{2}\int_{UN_f}\Vert\mathcal{W}_{\tilde{g}}\Vert^{n/2} \ d\Sigma.$$
Thus, from (\ref{W3}) and the fact that the $L^{n/2}$-norm of the Weyl tensor 
is conformally invariant,  we have that 
$$\int_{M^n}\Vert \mathcal{W}_g \Vert^{n/2}\ dM_g\geq c(n) \sum_{i=2}^{n-2}\beta_i(M^n;\F),$$
where the constant $c(n)$ is given by 
$c(n)=(\varepsilon(n))^{n/4}\Vol(\Sp^n)/2$. 

Now, assume that $$\int_{M^n}\Vert \mathcal{W}_g \Vert^{n/2}\ dM_g<c(n).$$ 
Then, it follows from (\ref{W3}) that $\sum_{i=2}^{n-2} \tau_i(f)<1.$ Thus, there 
exists $u\in \Sp^{n}$ such that the height function $h_u\colon M^n\rightarrow \R$ is a 
Morse function whose number of critical points of index $i$ satisfies $\mu_i(u)=0$ for 
any $2\leq i\leq n-2$. From the fundamental theorem of Morse theory 
(cf. \cite[Theorem 3.5, p. 20]{Milnor} or \cite[Theorem 4.10, p. 84]{CE}), it follows that $M^n$ has 
the homotopy type of a CW-complex with no cells of dimension $i$ for $2\leq i\leq n-2$. 
Hence, the homology groups satisfy $H_i(M^n;\Z)=0$ for all $2\leq i\leq n-2$. Moreover, since there 
are no $2$-cells, we conclude by the cellular approximation theorem that the inclusion 
of the $1$-skeleton $\mathrm{X}^{(1)}\hookrightarrow M^n$ induces isomorphism between 
the fundamental groups. Therefore, the fundamental group $\pi_1(M^n)$ is a free group on 
$\beta_1(M^n;\Z)$ elements and $H_1(M^n;\Z)$ is a free abelian group on $\beta_1(M^n;\Z)$ generators. 
In particular, if $\pi_1(M^n)$ is finite, then $\pi_1(M^n)=0$ and hence $H_1(M^n;\Z)=0$. 
From Poincar\'e duality and the universal coefficient theorem it follows that $H_{n-1}(M^n;\Z)=0$. 
Thus, $M^n$ is a simply connected homology sphere and hence a homotopy sphere. By the generalized 
Poincar\'e conjecture (Smale $n\geq 5$, Freedman $n=4$) we deduce that $M^n$ is homeomorphic to $\Sp^n$. \qed

\medskip

\noindent\emph{Proof of Theorem \ref{MHS}:} 
Let $f\colon (M^n,g)\rightarrow \Sp^{n+1}$ be an isometric minimal immersion of a compact $n$-dimensional
Riemannian manifold, with shape operator $A$. Since $f$ is minimal, the Ricci tensor of $M^n$ is given by 
$$
\Ric_g(X,Y)=(n-1)g(X,Y)-g(A^2X,Y),\ \ X,Y\in TM.
$$
Using the Gauss equation and (\ref{Rdec}), we obtain that
$$
\Vert\mathcal{W}_g\Vert^2=4\sum_{i<j} \left(\lambda_i\lambda_j+\frac{\lambda_i^2+\lambda_j^2}{n-2}-\frac{\Vert A\Vert^2}{(n-1)(n-2)}\right)^2,
$$
where $\lambda_1,\dots,\lambda_n$ are the principal curvatures of $f$.
After a straightforward computation, we find that
\begin{equation}\label{TW1}
\Vert\mathcal{W}_g\Vert^2=\gamma(n)\Vert A\Vert^4-\delta(n)\Vert \Ric_g-(n-1) g \Vert^2,
\end{equation} where 
$$
\gamma(n)=\frac{2(n^2-3n+5)}{(n-1)(n-2)}\ \ \text{and}\ \ \delta(n)=\frac{2(n+1)}{n-2}.
$$
From the Cauchy-Schwartz inequality we have that 
$$\Vert A\Vert^4\leq n\Vert A^2\Vert^2=n\Vert \Ric_g-(n-1) g \Vert^2.$$ 
Therefore, from (\ref{TW1}) we obtain 
$$\Vert\mathcal{W}_g\Vert^2\leq a(n)\Vert \Ric_g-(n-1) g \Vert^2,$$ where $a(n)=n\gamma(n)-\delta(n)$.
By integrating over $M^n$, it follows that
\be\label{InRic}
\int_{M^n}\Vert \mathcal{W}_g \Vert^{n/2}\ dM_g \leq  \big(a(n)\big)^{n/4}\int_{M^n} \Vert \Ric_g-(n-1) g \Vert^{n/2}\ dM_g.
\ee
The proof follows from (\ref{InRic}), Theorem \ref{main} applied to the composition of $f$ with the stereographic projection 
and the fact that the $L^{n/2}$-norm of the Weyl tensor is conformally invariant. Clearly, $c_1(n)=c(n)/\big(a(n)\big)^{n/4}$. \qed

\medskip

\noindent\emph{Proof of Theorem \ref{Th3}:} 
Assume that $(M^n,g)$ admits an isometric immersion into the sphere $\Sp^{n+1}$ or the hyperbolic space $\mathbb{H}^{n+1}$.
 Using the orthogonal decomposition of the curvature tensor $R_g$ in (\ref{Rdec}) 
and recalling the fact that $R_1=(1/2) g\varowedge g$, we obtain
$$
\Big\Vert R_g-\frac{\scal_g}{n(n-1)}R_1\Big\Vert^2=\Vert\mathcal{W}_g\Vert^2+
\Big\Vert\frac{1}{n-2}\Big(\Ric_g-\frac{\scal_g}{n}g\Big)\varowedge g\Big\Vert^2\geq \Vert\mathcal{W}_g\Vert^2.
$$
Thus, it follows that
\be\label{Rten}
\int_{M^n}\Big\Vert R_g-\frac{\scal_g}{n(n-1)}R_1\Big\Vert^{n/2}\ dM_g \geq \int_{M^n} \Vert\mathcal{W}_g\Vert^{n/2}\ dM_g.
\ee
Now the proof follows from (\ref{Rten}), Theorem \ref{main} and Remark \ref{rem}.  \qed

\bigskip

{\indent{\small\textsc{Department of Mathematics, University of Ioannina, 45110 Ioannina, Greece}}} \\
{\indent{\small {\it E-mail addresses:} \texttt{tvlachos@uoi.gr, chonti@cc.uoi.gr}}}

\end{document}